\documentclass{amsart}
\usepackage{verbatim,amssymb,amsmath,amscd,latexsym,amsbsy,mathrsfs} 
\usepackage{graphicx}
\input{xy}
\xyoption{all}

\newtheorem{thm}{Theorem}[section]
\newtheorem{cor}[thm]{Corollary}

\newtheorem{prop}[thm]{Proposition}
\newtheorem{lemma}[thm]{Lemma}
\newtheorem{rmk}[thm]{Remark}

\DeclareMathOperator*{\im}{im}

\DeclareMathOperator*{\Isom}{Isom}
\DeclareMathOperator*{\Diffeo}{Diffeo}
\DeclareMathOperator*{\Out}{Out}
\DeclareMathOperator*{\Aut}{Aut}
\DeclareMathOperator*{\InnerAut}{InnerAut}
\DeclareMathOperator*{\Stab}{Stab}
\DeclareMathOperator*{\Dom}{Dom}
\DeclareMathOperator*{\Hom}{Hom}

\newcommand {\cH} {\mathcal{H}}

\newcommand {\C} {{\mathbb C}}
\newcommand {\R} {{\mathbb R}}
\newcommand {\Z} {{\mathbb Z}}
\newcommand {\Q} {{\mathbb Q}}

\newcommand {\G} {{\mathbb G}}

\begin{document}
\title[Fibered fundamental groups]{Toward the structure of fibered fundamental groups of projective varieties}
\author{
        Donu Arapura    
}
 \thanks {Partially supported by the NSF }
\address{Department of Mathematics\\
 Purdue University\\
 West Lafayette, IN 47907\\
U.S.A.}
 \maketitle

 A useful dichotomy for groups is to subdivide
them  into large groups and small, where a  group is  large for our purposes if it surjects onto a nonabelian free group.
We want to study the large groups in the class
$\mathcal{P}$ of fundamental groups of complex smooth projective varieties.
Standard tricks of the trade, going back to Beauville, Catanese and
Siu \cite[chapter 2]{abc} and \cite[\S 5.1]{catanese}, show that such a group is  {\em fibered} in the sense that it
is given as an extension of an  orbifold group
\begin{equation*}
  \begin{split}
  \Gamma_{g;\vec{m}} =&\langle \alpha_1,\ldots, \alpha_{2g},\beta_1,\ldots,\beta_n\mid\\
&[\alpha_1,\alpha_2]\ldots [\alpha_{2g-1}, \alpha_{2g}] \beta_1\ldots \beta_n= \beta_1^{m_1}=\ldots \beta_n^{m_n}=1\rangle
\end{split}
\end{equation*}
 by a finitely generated group $K$. So we now come to the main
question that motivated this paper:
given an action $\Gamma_{g;\vec{m}}\to
\Aut(K)$, or perhaps only an outer action, with $K$ finitely generated, when can we expect the
semidirect product or some other associated extension to lie in
$\mathcal{P}$?  
  The group $\Gamma_{g;\vec{m}}$ will act on the finite
dimensional vector space $V=K/[K,K]\otimes \Q$. Let $G$ be the 
identity component of the Zariski closure of
the image of $\Gamma_{g;\vec{m}}$ in $GL(V)$.
We establish the
following necessary conditions for an extension of $\Gamma_{g;\vec{m}}$ by $K$
to lie in $\mathcal{P}$:
\begin{itemize}
\item The dimension of the  space of invariants $V^{\Gamma_{g;\vec{m}}}$  must  be even.
\item The $\Q$-algebraic group $G$ is 
  semisimple, and  the associated real group lies in the small list of the
groups arising from  Hermitian symmetric domains
of classical type \cite[p 518]{helgason}.
\item The last result applies  more generally to $V=H/[H,H]\otimes \Q$
  for any finite index  subgroup $H\subset K$. (A finite index subgroup $\Gamma'\subset \Gamma_{g,\vec{m}}$ will act on $V$,
  and $G$ can be defined as above using the $\Gamma'$-action.)
\end{itemize}
In the positive direction, we show that many  semisimple groups of classical
Hermitian type $G$ actually arise in this way from groups in $\mathcal{P}$.

Here is a more detailed summary of the contents of the paper. In the first
section, we construct a homomorphism $\rho:\pi_1(Y,y)\to O^+(X_y)$,
that we call nonabelian monodromy,
where $f:X\to Y$ is an oriented $C^\infty$ fibre bundle and  $O^+(X_y)=
\Out^+(\pi_1(X_y))$ is the group of orientation preserving outer
automorphisms of the fundamental group of  a fibre. When $Z=X_y$ is a curve, $O^+(X_y)$ is just
the mapping class group.
This has a well known representation given by its action on the
first homology of $Z$. More generally, a number of authors have
studied the action of subgroups of $O^+(Z)$ on the first homology of finite
(unramified) coverings of $Z$ \cite{gllm, koberda, loo}. All of these
extend to the more general situation, and we refer to these as
generalized Prym representations. In the second section, we study the
nonabelian monodromy of a family of smooth projective varieties. Our
main result here is that the Zariski closure of the image of the composite of a generalized Prym with
monodromy  is semisimple of classical Hermitian type. In a nutshell,
this is deduced from the fact that this is the monodromy
representation of a polarized variation  of Hodge structure of a
specific type. (And this is the main reason we work with
projective manifolds rather than compact K\"ahler manifolds.)
In the third section,  we deduce  the results stated in the first paragraph
by extending the monodromy theorem to families with singular fibres. In the
penultimate section, we go in a different direction. By combing the above
techniques with some work of Grunewald, Larsen,
Lubotzky, and  Malestein  \cite{gllm}, we compute Mumford-Tate groups
of some unamified covers of generic curves. The conclusion is that the
Hodge structure  of an  unramified cover looks very different from
the Hodge structure of the underlying curve.
The final section contains
examples, involving pencils of abelian varieties and Kodaira surfaces, 
with interesting monodromy groups.

The main ideas for this  paper were worked out during a visit to the
IH\'ES in the spring of 2015. My thanks to them
for a pleasant and productive stay. I would also like to thank one of the referees for bringing 
the very useful reference \cite{catanese} to my attention.

\section{Nonabelian monodromy}

Suppose that $F$ is a connected manifold.  Let $\delta$ be a path connecting
$x_0\in F$ to $x_1\in F$. A self diffeomorphism
$\phi:F\to F$, with $\phi(x_0)=x_1$, induces an automorphism $\pi_1(F,x_0)\to \pi_1(F,x_0)$
defined by $g\mapsto \delta^{-1}\phi_*(g)\delta$, where multiplication is taken in the fundamental groupoid. 
The corresponding outer automorphism is independent of
$\delta$. Now  suppose that  $f:X\to Y$ is a locally trivial $C^\infty$  fibre bundle with 
fibre $F$ and connected base $Y$. Then, after choosing a Riemannian metric on
$X$, we have a holonomy representation of
the fundamental group 
$$\tilde\rho:\pi_1(Y,y)\to \Isom(F)\subset \Diffeo(F)$$
to the group of isometries and therefore diffeomorphisms of $F$.  
Thus $\rho$ induces a homomorphism
$$\rho:\pi_1(Y,y)\to \Out(\pi_1(F,x_0))= \Aut(\pi_1(F,x_0))/\InnerAut(\pi_1(F,x_0))$$
We will refer to this as {\em nonabelian monodromy}.  

This can be described more topologically.
Given $\gamma\in \pi_1(Y)$, represent it by a $C^\infty$ map  $S^1\to Y$.
Then  we have an exact sequence
$$1\to \pi_1(F, x_0)\to \pi_1(X\times_Y S^1,x_0)\to \pi_1(S^1,0)=\Z\to 1$$
which necessarily splits (noncanonically). Let $ \tilde \gamma\in \pi_1(X\times_Y S^1)$ denote
a lift of $1\in \Z$.

\begin{lemma}
The outer automorphism of  $\pi_1(F)$ determined by $g\mapsto \tilde \gamma g \tilde\gamma^{-1}$
  coincides with   $\rho(\gamma)$.
\end{lemma}

\begin{proof}
We may replace $Y$ by $S^1$ and $X$ by $X\times_Y S^1$.  Let $\gamma\in \pi_1(Y)$ denote a generator.
 Let  us say that a $C^\infty$ path in $X$  is horizontal if its tangent vectors lie in $\ker df_x^\perp$. 
 Through any $x\in F$,  there is a  unique horizontal lift $\epsilon_x$ of
 $\gamma$ with initial point $\epsilon_x(0)=x$.    This is generally not closed. The holonomy  $\phi:F\to F$ sends $x$ to the end point  $\epsilon_x(1)$.
 Let $\delta$ be a path in $F$ connecting $x_0$ to $x_1=\phi(x_0)$. 
 The element $\epsilon_{x_0}^{-1}\delta\in\pi_1(X,x_0)$ maps to   $\gamma^{-1}$. So it must be conjugate to $\tilde \gamma^{-1}$.
 One easily checks that
 $$\rho(g)=\delta^{-1}\phi_*(g)\delta =  (\epsilon_{x_0}^{-1}\delta)^{-1}g(\epsilon_{x_0}^{-1}\delta)$$
 \end{proof}

\begin{cor}\label{cor:rho}
The outer action of $\pi_1(Y)$ on $\im \pi_1(F)$, by conjugation
via a set-theoretic section of $f_*$ in the sequence below
$$1\to \im \pi_1(F)\to \pi_1(X)\stackrel{f_*}{\longrightarrow} \pi_1(Y)\to 1, $$
is compatible with  $\rho$.  
\end{cor}

We set $O(F) = \Out(\pi_1(F))$.
If $F$ is oriented, then there is an induced orientation on $H^1(F,\R)= \Hom(\pi_1(F),\R)$.
Let $O^+(F)\subset O(F)$ denote the subgroup preserving
the orientation on $\Hom(\pi_1(F),\R)$. If $f$ is a  fibre bundle of oriented manifolds, then
the image of holonomy lies in the group of orientation preserving diffeomorphisms $\Diffeo^+(F)$.
It follows that $\rho(\pi_1(Y))\subseteq
O^+(F)$.
Let $O^+(F,\omega)=O(F,\omega)\cap O^+(F)$. When $F$ is a compact
oriented $2$-manifold, then $O^+(F)$ is just the mapping class group
\cite{fm}. Note that $O^+(F) = O(F,\omega)$, where $\omega\in
H^2(F)$ is the fundamental class.

Many of the familiar representations of the mapping class group
generalize to $O(F)$. The group $\Aut(\pi_1(F))$ has an obvious
 representation $\tau_\Z$ on $H_1(F,\Z)=
 \pi_1(F)^{ab}:=\pi_1(F)/[\pi_1(F),\pi_1(F)]$. Since inner
 automorphisms act trivially on $\pi_1(F)^{ab}$, $\tau_\Z$ factors through $O(F)$.
Let $\tau$ denote the corresponding rational representation $\tau_\Z\otimes\Q$.
In the case of the mapping class group, the kernel of $\tau$, called the Torelli group,
is rather large and somewhat mysterious.  Thus we  want to consider some
additional representations in order to detect elements of the Torelli group. 
It is convenient to adopt the viewpoint of \cite{gllm} that given a group $\Gamma$,
a representation $\sigma:\Gamma_1\to GL(V)$ of a finite index subgroup should be treated on the same
footing as a representation of $\Gamma$. We will refer to $\sigma$ as a {\em partial representation}  of $\Gamma$,
and  call $\Gamma_1$ the domain and denote it by $\Dom(\sigma)$. 
Let us say that two partial representations are {\em commensurable } if they agree after restriction
 to a finite index subgroup of the intersection of their domains.  
We can always induce a partial representation to an honest
representation, but it is better for our purposes not to do so.
We will mainly be concerned  with properties of 
 partial representations which depend only on the commensurability class,
 so we will occasionally  shrink the domains when it is convenient.
 Given $H\subset \pi_1(F)$  a subgroup of finite
 index, the stabilizer ${\Stab}(H)=\{\sigma\in \Aut(\pi_1(F))\mid \sigma(H)=
 H\}$, which  has finite index in $\Aut(\pi_1(F))$, acts on $H^{ab}\otimes \Q=H^{ab}_\Q$. Thus this is a partial representation of
 the automorphism group which we denote by $\tau^H$. If $H$ is characteristic, then $\Stab(H)=\Aut(\pi_1(F))$, so  $\tau^H$ is an honest
 representation.
 When $H$ is normal, then $G=\pi_1(X)/H$ acts on $H^{ab}$.
We can break the vector space $H_\Q^{ab}=H^{ab}\otimes \Q$ up into a sum
$\bigoplus_\chi (H_\Q^{ab})^\chi$ of  isotypic components
parameterized by the irreducible $\Q[G]$-modules $\chi$. In more explicit terms,
$(H_\Q^{ab})^\chi$ is the sum of all $\Q[G]$-submodules isomorphic to $\chi$.
This is a   representation of the subgroup
$\Stab(r)= \{\alpha\in \Aut(\pi_1(F))\mid \alpha\circ r=r\}\subseteq \Stab(H)$ where
$r:\pi_1(X)\to G$ denotes the projection. 
The family of partial 
representations obtained this way will be referred to as generalized
Prym representations, and denoted by $\tau^{H,\chi}$. In the case of the mapping class group, the
study of these representations (for nontrivial $H$) seems to have been
initiated by Looijenga \cite{loo}, and continued by Koberda
\cite{koberda}, and Grunewald, Larsen, Lubotzky, and  Malestein
\cite{gllm}. Koberbda \cite{koberda} has shown that $\bigoplus \tau^H$, as $H$ runs over characteristic subgroups, is a faithful 
representation of the based mapping class group. So in particular, the
nontriviality of elements of the based Torelli group can be detected
using these representations.
 
\begin{lemma}
 $\tau^H$ and $\tau^{H,\chi}$ descend to partial representations  of $\Out(F)$ after possibly shrinking their domains.
\end{lemma}

\begin{proof}
We focus on $\tau^H$, the argument for the second case is the same.
Let $V=H^{ab}_ \Q$  and let $S=\Dom(\tau^H)$.
 We have an exact sequence
 $$\pi_1(F)\to \Aut(\pi_1(F))\to O(F)\to 1$$
 Since inner automorphism by elements of $H$ act trivially on $V$, the image
 $\tau^H(\pi_1(F)\cap S)$  is finite. Since the image $\tau^H(S)$ is 
finitely generated  linear, and therefore residually finite,
we can choose a finite index subgroup of $\Gamma_1\subset S$ 
such that $\tau^H(\Gamma_1   \cap \pi_1(F))=\{1\}$.
\end{proof}

\section{Smooth projective families}
 Now suppose that $f:X\to Y$ is a smooth projective morphism of smooth
 varieties. We assume furthermore that  the fibres of $f$ are
 connected. By assumption, we have a relatively ample line bundle $\mathcal{L}$
 on $X$  with first Chern class $\omega$.
By Ehressman's theorem, $f$ is a $C^\infty$ fibre
 bundle. Thus we get a homomorphism $\rho:\pi_1(Y,y)\to O(X_y,\omega)$,
 where $X_y$ is the fibre over $y$.  It is worth observing that
 $\im\rho\subseteq O^+(X_y,\omega)$.
{\em From now on, 
the representation $\tau^H$ will
 denote the restriction of the previous $\tau^H$ to $O^+(X_y,\omega)$}. Let $n$ denote the dimension of $X_y$. 
 If $\pi:\tilde X_y\to X_y$ is the finite unramified covering corresponding to a finite index subgroup $H\subset \pi_1(X_y)$, then $X_y$ is projective with an ample class $\tilde \omega =\pi^*\omega$.
 By the hard Lefschetz theorem, we have a symplectic
 form
$$(\, , \,):H^1(\tilde X_y,\Q)\times H^1(\tilde X_y,\Q)\to H^2(\tilde X_y,\Q)\stackrel{\cup
  \tilde \omega^{n-2}}{\longrightarrow} H^{2n}(\tilde X_y,\Q)\cong \Q$$
This induces a dual pairing denoted by the same symbol on
$H_1(\tilde X_y,\Q)\cong H_\Q^{ab}$. The action of $\Stab(H)$ preserves this, so $\tau^H$
is a representation into the corresponding symplectic group.

In order to analyze the Zariski closures of these representations, 
we need to recall some basic facts about Mumford-Tate groups;
we refer to  \cite[\S1]{milne1} or \cite{moonen} for a more detailed treatment.
Recall that a rational Hodge structure $H$
consists of  a finite dimensional $\Q$-vector space and a
decomposition $H_\C=H\otimes \C=\oplus H^{pq}$ with $H^{qp}=\bar H^{pq}$.
The bigrading determines and is determined by the homomorphism of $h:\C^*\to GL(H_\R)$, given
by $h(\lambda)v=\sum \lambda^{q}\bar\lambda^pv^{pq}$.  
The Mumford-Tate group $MT(H)\subseteq GL(H)$ is the smallest $\Q$-algebraic subgroup
whose real points contain the image of $\C^*$. For our purposes, it is more convenient to work with  a slightly smaller group
called the Hodge group or the special Mumford-Tate group $SMT(H)$ given as the identity component of
  $MT(H)\cap SL(H)$.
This is the smallest $\Q$-algebraic subgroup $SMT(H)\subset GL(H)$, whose real points contain  the
image of the unit circle $h(U(1))$.

Let us say that a real algebraic group $G$ is of Hermitian type if it
is connected, reductive and the quotient of it  by a maximal compact
subgroup $K$ is a Hermitian
symmetric space. Say that $G$ is of symplectic Hermitian type if in addition
$G/K$ has a totally geodesic holomorphic embedding  into a Siegel
upper half plane.  By Cartan's classification, a noncompact simple group of
Hermitian type is isogenous to $SU(p,q), SO(2,p)^o, SO^*(2p),
Sp(2g,\R)$ or certain real forms of $E_6$ or $E_7$ \cite[p 518]{helgason}. By
Satake \cite{satake}, only the first four are symplectic.
A $\Q$-algebraic group $G$ will  be called (symplectic) Hermitian if
$G(\R)$ has these properties.

\begin{thm}[Mumford]\label{thm:mumford}
 Suppose that  $H$ is polarizable  of
type $\{(-1,0), (0,-1)\}$; in other words, suppose that $H$ is the  first
homology of an abelian variety.  Then  $M=SMT(H)$ is of symplectic Hermitian type.
\end{thm}

\begin{proof}
  This is stated in Mumford \cite[pp 348-350]{mumford} without proof, so we
 give a brief explanation here.
The connectedness of $M$ is clear from the definition. The polarizability of $H$
shows that $M$ leaves a positive definite form invariant, and this
implies reductivity. The
Hermitianness of $M$ can be deduced from \cite[thm 1.21] {milne} (the
homomorphism $h$ satisfies conditions (a), (b), (c) of that
theorem). Furthermore, since $H$ has a polarization $\psi$, we have a
homomorphism $U(1)\to Sp(H,\psi)$, whence an inclusion $M\subset
Sp(H,\psi)$ satisfying the $(H_2)$ condition of
\cite{satake}. Therefore, we have an embedding of  the symmetric space
associated to $M$ into the symmetric space associated to $Sp(H,\psi)$.
\end{proof}

As noted above, the result  puts very strong restrictions on the
possible values for $M$. Note that $SMT(H)= SMT(H^*)$.
So we  switch to the dual when it is convenient.

Here is the first main result.

\begin{thm}\label{thm:main1}
Suppose that $X\to Y$ is a  smooth projective family with ample class
$\omega$ over a smooth
quasiprojective base. Let $\rho:\pi_1(Y,y)\to O(X_y,\omega)$ denote the
nonabelian monodromy. Then for any finite index  normal subgroup $H\subset \pi_1(X_y)$ and
character $\chi$ of the quotient,
  the identity component of the  Zariski closure of the image of
  $\tau^{H,\chi}\circ \rho$ is semisimple of symplectic Hermitian type.
\end{thm}

The proof will rely on the following   lemmas.

\begin{lemma}\label{lemma:gt}
Suppose that  we are given a homomorphism of groups $r:\Gamma\to G$  and an action of
another group $\Pi$ on $\Gamma$ preserving $r$, i.e. $\Stab(r)
=\Pi$. Then $r$ extends to a homomorphism $\tilde r:\Gamma\rtimes \Pi
\to G$.
\end{lemma}

The proof is
immediate from the standard formulas for the semidirect product.

\begin{lemma}\label{lemma:key}
With the same notation as in theorem \ref{thm:main1}, there exist a
  surjective generically finite  morphism of smooth varieties $p:\tilde Y\to Y$ such that
  $p_*(\pi_1(\tilde Y))\subset \pi_1(Y)$ has a finite index and is
  contained in the domain of $\tau^{H,\chi}$. Furthermore
  $\tau^{H,\chi}\circ \rho\circ p_*$ is the monodromy representation of a polarizable
  variation  of Hodge structures on $\tilde Y$ of type $\{(-1,0), (0,-1)\}$.
\end{lemma}

\begin{proof}
Let $\Gamma=\pi_1(X_y)$ and 
 let $r:\Gamma\to G=\Gamma/H$ denote projection. Let $e\in \Q[G]$ be the central idempotent
whose image is $A_\chi$.
After passing to an \'etale cover $p: Y_1\to Y$, we can assume
$\rho(p_*\pi_1(Y_1))\subseteq
\Dom(\tau^{H,\chi})$.  The generic fibre of $X\times_Y Y_1\to
Y_1$ admits a rational point defined over some finite extension $K$ of the
function field $\C(Y_1)$. Let $\tilde Y$ be a desingularization of the normalization of $Y_1$
in $K$. Let $\tilde X= X\times_Y \tilde Y$. Then the map $\tilde X\to
\tilde Y$ possesses a section. Therefore $\pi_1(\tilde X)$ is the
semidirect product $\Gamma\rtimes \pi_1(\tilde Y)$.
Let $C\subset Y$ be a curve given as a complete intersection of ample divisors in
general position. Let $\tilde C\subset \tilde Y$ denote an irreducible
component of the preimage of $C$.  Let $U$ be the complement of the
set of branch points of  $\tilde C\to C$, and let $\tilde U\subset
\tilde C$ denote the preimage.  Consider the diagram
$$
\xymatrix{
 \pi_1(\tilde U)\ar[r]\ar[d]^{\alpha} & \pi_1(\tilde Y)\ar^{p_*}[d] \\ 
 \pi_1(U)\ar[r]^{\beta} & \pi_1(Y)
}
$$
By a suitable Lefschetz hyperplane theorem \cite[p 153]{gm}, we obtain
a surjection $\pi_1(C)\to \pi_1(Y)$.  We have a surjection
$\pi_1(U)\to \pi_1(C)$. Combining these two assertions
shows that $\beta$ is surjective. Covering  space theory shows the
image of $\alpha$  has
finite index in $\pi_1(U).$  After  chasing the diagram  the other way, we can conclude $p_*\pi_1(\tilde Y)$ has
finite index in $\pi_1(Y)$.

Applying lemma \ref{lemma:gt} yields normal subgroup $T\subset
\pi_1(\tilde X)$ with $\pi_1(\tilde X)/T=G$. Let $p:Z\to \tilde X$ be the corresponding
Galois \'etale cover.
 Then $\tau^{H}\circ \rho$ is
the monodromy representation of the local system $\bigcup_y
H_1(Z_y,\Q)$ which can be identified with
$R^1 (f\circ p)_*\Q^\vee$. The latter clearly underlies  a polarized variation
of Hodge structure of  type $\{(-1,0), (0,-1)\}$.  Note that $G$ acts
on this by automorphisms. The representation
$\tau^{H,\chi}\circ \eta$ corresponds to the sub variation of Hodge
structure $e(R^1 (f\circ p)_*\Q^\vee)$.
\end{proof}

\begin{proof}[Proof of theorem \ref{thm:main1}]

Let $Z$ be the identity component of the  Zariski closure of the image
of $\tau^{H,\chi}\circ \rho$, and let $\cH$ be the corresponding variation of Hodge structure.
By the theorem of Andr\'e \cite[\S 5 thm 1]{andre}, $Z$ is a normal
subgroup of the derived group $DMT(\cH_x)$ of the  Mumford-Tate group of a very
general fibre $\cH_x$. Observe that $DMT(\cH_x)=DSMT(\cH_x)$ and this is isogenous to $SMT(\cH_x)$
because the last group is semisimple.
 Therefore   $DSMT(\cH_X)$ is
semisimple of  symplectic Hermitian type by theorem \ref{thm:mumford}.
Therefore $Z$ also has the same property.
\end{proof}

In the previous set up, given  $\gamma\in \pi_1(Y)$ some positive power
 $\gamma^n$ will lie $\Dom(\tau^H)$. By the same technique, we get
 further constraints.

\begin{prop}
With the same notation as in theorem~\ref{thm:main1},
  suppose that  $\gamma\in \pi_1(Y)$ is a loop around a
  smooth boundary divisor of some smooth compactification. Then
  $\tau^{H,\chi}\circ \rho(\gamma^n)$ is  quasi-unipotent for all $n$ as above.
\end{prop}

\begin{proof}
After replacing $n$ by a multiple, we can assume that $\gamma^n\in
p_*\pi_1(\tilde Y)$, where $p:\tilde Y\to Y$ is as in lemma~\ref{lemma:key}. 
  The result  now follows from \cite[lemma 4.5]{schmid}.
\end{proof}

\section{Fibered fundamental groups}

Let $f:X\to Y$ be a projective map of smooth quasiprojective varieties
such that $f$ has connected fibres. Then we have a surjection
$\pi_1(f):\pi_1(X)\to \pi_1(Y)$. The kernel of this map may be quite
large, however. We first want to factor this through a map with better
properties. Given a divisor $D\subset Y$ with simple normal
crossings, the restriction $\pi_1(Y-D)\to \pi_1(Y)$ is surjective. The
kernel is the normal subgroup generated by loops $\gamma_i$ about the
components $D_i$. If $m_i>1$ are integers, define the orbifold
fundamental group $\pi_1^{orb}(Y, \sum m_iD_i)$ as the quotient of $\pi_1(Y-D)$
be the normal subgroup generated by $\gamma_i^{m_i}$. This can be
interpreted as the fundamental group of $Y$ with a suitable orbifold
structure, but we won't need this. After removing a closed subset 
$Z\subset Y$ of codimension at least $2$, we can suppose that the discriminant
of $f$ is a smooth divisor $D=\sum D_i$, and
that $f^{-1}D$ is a divisor with normal crossings  such that the
restriction of $f$ to  the intersections
of components are submersions over $D$. Let $m_i$ denote the greatest
common divisor of the multiplicities of the components of $f^{-1}D_i$.
The following is proved in \cite[lemma 3.5]{arapura}.

\begin{prop}\label{prop:orb}
  Let $y_0\in Y-D-Z$. Then $\pi_1(f)$ factors  through a surjection
  $\phi:\pi_1(X)\to \pi_1^{orb}(Y,\sum m_iD_i) $ such that
$$
\xymatrix{
 \pi_1(f^{-1}(y_0))\ar[d]^{r_1}\ar[r] & \pi_1(X-f^{-1}(D\cup Z))\ar[d]^{r_2}\ar[r]^{\psi} & \pi_1(Y-D-Z)\ar[d]^{r_3}\ar[r] & 1 \\ 
 \ker(\phi)\ar[r] & \pi_1(X-f^{-1}Z)\ar[r]^{\phi} & \pi_1^{orb}(Y,\sum m_iD_i)
 \ar[r]  & 1
}
$$
commutes and has exact rows.
The map  $r_1:\pi_1(f^{-1}(y_0))\to \ker\phi$ is surjective. In
particular, $\ker \phi$ is finitely generated.
\end{prop}

If  $f$ is flat then $f^{-1}Z$ has codimension
$\ge 2$. Consequently $\pi_1(X-f^{-1}Z)\cong \pi_1(X)$ et cetera. 

\begin{cor}
  Assuming flatness of $f$,  $Z$ can be omitted in the statement of the proposition.
\end{cor}

Therefore if $f$ is flat, we have an exact sequence
$$1\to \ker(\phi)\to \pi_1(X)\to \pi_1^{orb}(Y)\to 1$$
with finitely generated kernel, 
where we write $\pi_1^{orb} (Y)=\pi_1^{orb}(Y,\sum m_iD_i)$ for
simplicity. This gives an outer  action of $\rho: \pi_1^{orb} (Y)\to \Out(\ker\phi)$.
Given a finite index subgroup $H\subset \ker(\phi)$. Let $\sigma^H$
denote the partial representation of $\Out(\ker\phi)$  on
$H^{ab}_\Q$ with domain $\Stab(H)$. We write $\sigma =\sigma^H$ when
$H=\ker(\phi)$ is the full group.

\begin{prop}\label{prop:phi} 
  With the assumptions and notation of the previous paragraph, the identity component
  of the Zariski closure of the image of $\sigma^H\circ \rho$ is semisimple  of
  symplectic Hermitian type.
\end{prop}

\begin{proof}
From proposition \ref{prop:orb}, we deduce a diagram
$$
\xymatrix{
 1\ar[r]&\im\pi_1(X_{y_0})\ar[d]^{\bar r_1}\ar[r] & \pi_1(X-f^{-1}D)\ar[d]^{r_2}\ar[r]^{\psi} & \pi_1(Y-D)\ar[d]^{r_3}\ar[r] & 1 \\ 
1\ar[r]& \ker(\phi)\ar[r] & \pi_1(X)\ar[r]^{\phi} & \pi_1^{orb}(Y)
 \ar[r]  & 1
}
$$
where $\bar r_1$ is surjective.
Corollary \ref{cor:rho} shows that nonabelian monodromy on $\im \pi_1(X_{y_0})$ concides with the
conjugation action of $\pi_1(Y-D)$ coming from the first row above.
  Let $K\subset \im\pi_1(X_{y_0})$ be the preimage of $H$ under the map $\bar r_1$.
  Then $K^{ab}_\Q$ surjects onto $H^{ab}_\Q$, and this is compatible with the partial actions
of $\pi_1(Y-D)$ and $\pi_1^{orb}(Y)$ given by conjugation.

   Therefore the Zariski closure of the image of 
  $\sigma^H\circ \rho$ is a quotient of the Zariski closure of  the image of 
$\Dom(\tau^K)\subseteq \pi_1(Y-D)$ in $GL(K^{ab}_\Q)$. A quotient of a  semisimple
group of symplectic Hermitian type is again semisimple of symplectic  Hermitian type. 
Therefore the proposition follows from theorem \ref{thm:main1}.
\end{proof}

For the remainder of this section, we focus on the case where $Y$ is a smooth projective
curve of genus $g$.
Its fundamental group is given by
$$\Gamma_g=\langle \alpha_1,\ldots, \alpha_{2g}\mid
[\alpha_1,\alpha_2]\ldots [\alpha_{2g-1},\alpha_{2g}]=1\rangle$$
Given  integers $m_1,\ldots, m_n>
1$, let
\begin{equation*}
\begin{split}
 \Gamma_{g; m_1,\ldots, m_n} = &\Gamma_{g;\vec{m}}=
\langle \alpha_1,\ldots, \alpha_{2g},\beta_1,\ldots,\beta_n\mid\\
&[\alpha_1,\alpha_2]\ldots [\alpha_{2g-1}, \alpha_{2g}] \beta_1\ldots \beta_n= \beta_1^{m_1}=\ldots \beta_n^{m_n}=1\rangle
\end{split}
\end{equation*}
  This is $\pi_1^{orb}(Y,\sum m_ip_i)$ for some $p_i\in Y$.
 By \cite{fox}, there exists a torsion free normal subgroup
 $\Gamma'\subset \Gamma_{g;\vec{m}}$ of finite index. In more  geometric terms, 
 $\Gamma'$  is the ordinary fundamental group of a  curve $Y'$, where $r:Y'\to Y$  is a 
Galois cover with ramification divisor $r^*(\sum m_ip_i)$.
Consequently $\Gamma'=\Gamma_h$ for some $h$, and  $2h-2$ is a positive integer multiple of
$$2g-2+\sum \frac{m_i-1}{m_i}$$
Let us say that $\Gamma_{g;\vec{m}}$ is hyperbolic if the above expression is greater
than zero.

\begin{lemma}\label{lemma:HGamma}
  The following statements hold.
  \begin{enumerate}
  \item[(a)] $\dim H^1(\Gamma_{g;\vec{m}},\Q) =2g$
\item[(b)] $H^2(\Gamma_{g;\vec{m}},\Q) \cong H^2(\Gamma',\Q)$ is one dimensional.
  \end{enumerate}
\end{lemma}

\begin{proof}
  The first statement follows immediately from the presentation. 
Let $G=\Gamma/\Gamma'$.
The  Hochschild-Serre spectral sequence gives an isomorphism
$$H^2(\Gamma_{g;\vec{m}},\Q) \cong H^2(\Gamma',\Q)^{G}$$
The generator of  $H^2(\Gamma',\Q)=H^2(Y',\Q)$ is invariant under $G$.
This proves (b).
\end{proof}

Suppose that $f:X\to Y$ is surjective holomorphic map from a
 smooth projective variety.   Recall that a fibre is a multiple fibre
 if the greatest common divisor of the multiplicities of the
 components is greater than $1$.
Suppose $f$ has $n$ multiple fibres with multiplicity $m_i$. Then from the previous
discussion, we obtain a surjective homomorphism $\phi:\pi_1(X)\to \Gamma_{g;\vec{m}}$
with finitely generated kernel.

\begin{thm}[Catanese {\cite[thm 5.14]{catanese}}]\label{thm:cat}
  Conversely, any surjective homomorphism  $\phi:\pi_1(X)\to \Gamma_{g;m_1,\ldots, m_n}$ with
  finitely generated kernel must  arise in the above manner from a holomorphic map $X\to Y$ to a genus $g$ curve
  with exactly $n$ multiple fibres of muliplicity $m_1,\ldots m_n$.
\end{thm}

We can now prove the main results announced in the introduction.
If  $X$ is a smooth projective variety such that $\pi_1(X)$ surjects onto a nonabelian free group, then
by \cite[cor 5.4, prop 5.13]{catanese} it surjects onto  a  hyperbolic $\Gamma_{g;\vec{m}}$ with finitely generated kernel.
In fact, the hyperbolicity condition is not needed for the results below.

\begin{thm}\label{thm:main2}
Let $X$ be a smooth complex projective variety.
  Suppose that $\phi:\pi_1(X)\to \Gamma_{g;\vec{m}}$ is a surjective
  homomorphism such that  $\ker\phi$ is finitely generated.
For any finite index subgroup $H\subset \ker(\phi)$, the identity component of the Zariski closure of the image of 
  $\sigma^H\circ \rho $ is 
   semisimple  symplectic Hermitian, where
   $\rho: \Gamma_{g;\vec{m}}\to \Out(\ker \phi )$ is the representation associated to the extension.
\end{thm}

\begin{proof}
  This follows from proposition \ref{prop:phi}, and theorem \ref{thm:cat}.
\end{proof}

\begin{cor}
 The partial representation $\sigma^H\circ \rho $ is semisimple.
\end{cor}

\begin{proof}
The theorem implies that the  Zariski closure has a compact real form. So the corollary follows from Weyl's unitary trick.
\end{proof}

\begin{rmk}
 The groups $\Gamma_{g,\vec{m}}$ which are not hyperbolic are either finite or abelian. The theorem is vacuous in the finite case.
 In the abelian  case, the theorem implies that it acts through a finite quotient.
\end{rmk}

\begin{prop}\label{prop:even}
Let $X$ be a smooth projective variety. Given an exact sequence 
$$1\to K\to \pi_1(X)\to \Gamma_{g,\vec{m}}\to 1$$
with $K$ finitely generated, $\dim V^{\Gamma_{g;\vec{m}}}$ is even, where $V=
K/[K,K]\otimes \Q$.
\end{prop}

\begin{proof}
Let $\Gamma=\Gamma_{g,\vec{m}}$.
From the Hochschild-Serre spectral sequence, we deduce an exact
sequence
$$0\to H^1(\Gamma,\Q)\to H^1(\pi_1(X),\Q)\to H^0(\Gamma, H^1(K,
\Q))\to H^2(\Gamma,\Q)\to H^2(\pi_1(X),\Q)$$
By  theorem \ref{thm:cat}, the map $\pi_1(X)\to \Gamma$ is
realized by a surjective holomorphic map $f:X\to Y$ to a curve.
We claim that  $H^2(\Gamma,\Q)\to H^2(\pi_1(X),\Q)$ is
injective. This together with lemma \ref{lemma:HGamma}, will imply
that the Betti number $b_1(X)= 2g+ \dim H^0(\Gamma, H^1(K,
\Q))$. It would follow that the dimension of coinvariants $\dim V_{\Gamma} =\dim H^0(\Gamma, H^1(K,
\Q))$ is even. Since $\Gamma$ acts semisimply, this is also the dimension of the space of invariants.

To prove the claim, choose  $\Gamma'\subset\Gamma$ as above. Then $\Gamma'$ is the
(usual) fundamental group of some finite Galois cover $Y'\to Y$. Let
$X'$ be a desingularization of  $X\times_Y Y'$. Since by lemma \ref{lemma:HGamma}, 
$$H^2(\Gamma,\Q)\cong H^2(\Gamma',\Q)\cong H^2(Y',\Q)$$
it suffices to prove to prove that the composite of
$$H^2(\Gamma',\Q)\to H^2(\pi_1(X'),\Q)\to H^2(X',\Q)$$
is injective. But this is clear because the image contains the
fundamental class of the fibre of $X'\to Y'$.

\end{proof}

\begin{rmk}
 The same proof shows that $\dim V_{\Gamma_{g;\vec{m}}}$ is even,  when $X$ is compact K\"ahler.
\end{rmk}

\section{\'Etale covers of general curves}

Let $\Gamma_g$ be the fundamental group of a genus $g$ curve. Fix a
finite index normal subgroup $H\subset \Gamma_g$ with quotient
$G$. This determines a Galois \'etale cover $\tilde C\to C$ of any
curve of genus $g$. Our goal is to compute the special Mumford-Tate
group of
$H^1(\tilde C)$ when $C$ is  {\em very general curve}. This means that
$C$ occurs in the complement of a countable union of proper subvarieties of the moduli
space of curves $M_g(\C)$. Since $G$ acts on $H^1(\tilde C)$, we can decompose
it as sum of isotypic components $\bigoplus_\chi H^1(\tilde C)^\chi$, as $\chi$ runs over (isomorphism classes of) 
irreducible $\Q[G]$-modules. It suffices to compute $SMT(H^1(\tilde C)^\chi)$.

In order to state the main result, we need to recall some terminology from \cite{gllm}.
The space $V=H^1(\tilde C,\Q)$ carries a
$\Q[G]$-valued pairing  given by
$$\langle u,v\rangle = \sum_{g\in G} (u, gv)g$$
where $(,)$ is the usual intersection pairing on $V$.
This is sesquilinear and skew-Hermitian with respect to the involution $g^*=g^{-1}$, i.e. $\langle gu, hv\rangle
=g \langle u,v\rangle h^*$ and $\langle u,v\rangle = -\langle
v,u\rangle^*$ \cite[lemma 3.1]{gllm}. We have
$$\im \tau^H\subseteq \Aut(V, \langle, \rangle)$$
We can break up the group on the right into simpler pieces.
Let $\chi$ be the class of an irreducible $\Q[G]$-module in the
Grothendieck group. The $\chi$-isotypic submodule of $\Q[G]$ is a
subalgebra $A_\chi$ which is a matrix algebra over a division ring
$D_\chi$. The algebra $A_\chi$ is stable under $*$ \cite[lemma 3.2]{gllm}.  Let $L_\chi$ denote the center of $A_\chi$, and let $K_\chi$ the fixed field  for the involution.
 The isotypic component $V^\chi= H^{ab,\chi}_\Q$  becomes a $A_\chi$-submodule which is also stable under the action of $\Stab(r)$. 
The restriction of the pairing $\langle, \rangle$ to $V_\chi$
is $A_\chi$-valued \cite[lemma 3.3]{gllm}.  
The group  $\Aut(V_\chi, \langle, \rangle)$ is  naturally an algebraic group over $K_\chi$,
but we wish to regard it  as an algebraic groups over $\Q$. More
formally, we apply  Weil restriction $\G_{H,\chi}=
Res_{K_\chi/\Q}\Aut(V_\chi, \langle, \rangle)$. We will need to consider
the subgroup $\G_{H,\chi}^1\subset \G_{H,\chi}$ of elements with
reduced norm equal to $1$.
The associated complex group $\G_{H,\chi}(\C)$ is symplectic,
orthogonal or general linear according to whether
$A_\chi\otimes_{K_\chi} \R$ becomes a matrix algebra over $\R,\C$ or
the quaternions. Proofs of this and more can be found in \cite{gllm}.
Let $\tau^{H,\chi}$ denote the representation corresponding to $V^\chi$.
The image of this is  in $\G_{H,\chi}(\Q)$.
We can decompose
$\tau^H= \sum_\chi \tau^{H,\chi}$. 
We say that  the quotient map
$r:\Gamma_g\to G$  is redundant if it
factors as
 $$\Gamma_g\stackrel{r'}{\to} F_g\stackrel{r''}{\to} G$$
where $F_g$ is a free group on $g$ generators,
 $r'$ is a surjection
  and, $r''$  contains a free
  generator in its kernel.  Clearly, there is a redundant homomorphism
  onto $G$ if and only if it is generated by fewer than $g$ elements.

\begin{thm}\label{thm:main3}
  Let $g\ge 3$. Suppose that $r:\Gamma_g\to G$ is a redundant surjective
  homomorphism. Let $C$ be a
  very general curve of genus $g$ and $\tilde C$ is the corresponding
  \'etale $G$-cover. For each irreducible $\Q[G]$-module $\chi$, the special Mumford-Tate group of the
  isotoypic component $SMT(H^1(\tilde C)^\chi)=\ \G_{H,\chi}^1$

\end{thm}

The main ingredient is the following theorem of Grunewald, Larsen,
Lubotzky, and  Malestein \cite[thm 1.6]{gllm}.

\begin{thm}\label{thm:gllm}
  Suppose that $g\ge 3$ and  $r$ is redundant. Then for any  irreducible
  $\Q[G]$-module $\chi$, $\im \tau^{H,\chi}$ is an
  arithmetic subgroup of $\G_{H,\chi}^1$. 
\end{thm}

\begin{lemma}\label{lemma:mt}
Let $H, G, r, \chi$ be as above, but with $r$ not necessarily redundant.
Let $C$ be a smooth projective curve of genus $g$ and $\tilde C$ the
corresponding unramified $G$-cover. The
$\chi$-isotypic component of $H_1(\tilde C,\Q)=H^1(\tilde C, \Q)^\vee$ is a sub Hodge structure
with special Mumford-Tate group contained in $\G_{H,\chi}^1$.
\end{lemma}

\begin{proof}
  Since $G$ acts holomorphically on $\tilde C$, it preserves the
  canonically polarized Hodge
  structure $H^1(\tilde C)$. Therefore the $\chi$-isotypic component $M\subset H_1(\tilde
  C,\Q)$ is a polarized sub Hodge structure.  Let $K=K_\chi$, $A=A_\chi$,
  $\mathcal{G}=\Aut(M,\langle,\rangle)$ and $\mathcal{G}^1\subset
  \mathcal{G}$ subgroup of elements with reduced norm $1$. Then
$\G_{H,\chi}^1=Res_{K/\Q}\mathcal{G}^1$. 
To prove the lemma, it suffices to
  show that the image of  $h:U(1)\to GL(M_\R)$  lies
  $Res_{K/\Q}\mathcal{G}^1(\R)=\mathcal{G}^1(K\otimes \R)$. Since as noted $G$ acts by
  automorphisms of the polarized  Hodge structure, it follows that the image of
  $h:U(1)\to GL(M_\R)$  lies in $\Aut_G(M_\R,\langle,
  \rangle)=\Aut_A(M_\R,\langle, \rangle)$.  This
  almost does it but it remains to check the reduced norm condition. After
  extending scalars to $\C$, we see that the reduced norm of
  $h(\lambda)$ equals
$$\prod \lambda^{(q-p)\dim M^{pq}}=1$$

\end{proof}

\begin{proof}[Proof of theorem \ref{thm:main3}]
  Let $Y=M_{g}[n]$ be the moduli space of curves of genus $g$ with level
  $n\ge 3$ structure \cite{mf}. This is a fine moduli space, so it
  carries a universal curve $X\to Y$. We can identify $\pi_1(Y)$ with
  congruence subgroup $\ker[Mod_g\to Sp_{2g}(\Z/n\Z)]$ of the mapping class group $Mod_g$.
Lemma~\ref{lemma:key} gives a
  surjective map $p:\tilde Y\to Y$ such that $\pi_1(\tilde Y)$ has
  finite index in $\pi_1(Y)$ and such that $\tau^{H,\chi}\circ p_*$ comes
  from a polarizable variation of Hodge structure $\cH$.  
Lemma~\ref{lemma:mt} gives an inclusion $SMT(\cH_y)\subseteq
\G^1(\Q)$ for any $y\in \tilde Y$. For very general $y\in
  \tilde Y$,  the identity component $Z$ of the Zariski closure of the monodromy
  group of $\cH$ lies in $SMT(\cH_y)$ \cite[thm 6.19]{milne}.
By theorem~\ref{thm:gllm}, we
  have $Z=\G^1(\Q)$.  Thus we obtain the reverse inclusion
  $\G^1(\Q) \subseteq SMT(\cH_y)$ for very general $y$.
\end{proof}

When $G=\Z/2\Z$, we recover the main result of \cite{bp} that
the special Mumford-Tate group of a very general Prym variety is the
full symplectic group. 

We record the following corollary of the proof for later use.

\begin{cor}\label{cor:mt}
Assume $g\ge 3$ and that $H\subset \Gamma_g$ is a finite index normal
subgroup such that $\Gamma_g/H$ is generated by fewer than $g$ elements.
Then  the identity component of the Zariski closure of $\im \tau^{H}$ is $\prod_\chi \G_{H,\chi}^1$.
\end{cor}

\section{Examples}

\subsection{Families of abelian varieties}
Theorems \ref{thm:main1} and \ref{thm:main2} give strong restrictions
on the representations $\tau\circ \rho$ and $\sigma\circ\rho$
associated to  fundamental
groups of varieties. We now want  to understand which representations can
actually arise in this way. To simplify the
 statement of the proposition below, let us say that a group $\Gamma$ occurs in
theorem \ref{thm:main1} (respectively theorem \ref{thm:main2}) if 
 it is isomorphic  to the image of $\tau\circ
\rho$ (respectively $\sigma\circ\rho$) for an  example
satisfying the conditions of the theorem.

\begin{prop}\label{prop:conv}
Let $\Gamma_1$ be an arithmetic subgroup of a  special
Mumford-Tate group $G$ of a polarized Hodge structure of type $\{(-1,0),
(0,-1)\}$. Then all but  finitely many finite index subgroups 
$\Gamma\subset \Gamma_1$ occur in theorem \ref{thm:main1}. If none of the irreducible factors of the symmetric space
$G(\R)^0/K$ are the one ball, two ball, or genus $2$
   Siegel upper half plane, then almost all finite index
   $\Gamma\subset \Gamma_1$ occur in theorem \ref{thm:main2}. Certain
   lattices in $Sl_2(\R)$ occur in theorem \ref{thm:main2}. 
\end{prop}

Before we explain the examples, we briefly recall the notion of  a Shimura variety of Hodge type  following
the  original viewpoint of Mumford \cite{mumford, mumford2}. Let $G$
be the  special Mumford-Tate group of an abelian variety, and
$\Gamma\subset G$ an arithmetic subgroup. Then, roughly
speaking,  the Shimura variety $S$ parameterizes  the set of all abelian varieties having
Mumford-Tate group finer than or equal to   $G$ such that $\Gamma$
fixes the lattice. Here ``finer'' should be
understood as ``is contained in''. To be precise, let us fix a
polarized Hodge structure $L=H_\Z$ of type $\{(-1,0), (0,-1)\}$ together
with a polarization $\psi$. Let
$G =SMT(H)$ along with the homomorphisms 
$h:U(1)\to G(\R)$ and  $\rho:G\to Sp(H_\Q,\psi)$ that come with it.
Given this,  $K = \{g\in G(\R)\mid \rho(g) h = h\rho(g)\}$ is a
maximal compact subgroup of $G(\R)$ such that $\tilde S=G(\R)/K$ is Hermitian
symmetric domain which embeds into $\mathbb{H}_q=Sp(H_\R,\psi)/\{\text{centralizer
    of $h$}\}$. The space $\mathbb{H}_q$  can identified with the Siegel
  upper half plane of genus $q=\dim H/2$. Thus  $\tilde S$ carries a holomorphic family of
  abelian varieties $\mathcal{U}\to \tilde S$ given by pulling back the universal family from
  $\mathbb{H}_q$. Given an arithmetic subgroup $\Gamma_1\subset
  G(\Q)$, we can choose a finite index torsion free  subgroup
  $\Gamma\subset\Gamma_1$ stabilizing $L$.
Let $S=\Gamma\backslash \tilde S$.
 Then we can construct a family of Abelian
  varieties $\Gamma\backslash \mathcal{U}\to S$.
We note that $S$ is quasiprojective,  with a minimal  projective
compactification $\bar S$ constructed by Baily-Borel \cite{bb}. The following
 is almost immediate from the construction.

 \begin{lemma}
The  given action of
   $\Gamma$ on $L$ is the monodromy of $\Gamma\backslash
   \mathcal{U}\to S$ on the first homology of the fibre. The fundamental group $\pi_1(\Gamma\backslash
   \mathcal{U})=L\rtimes \Gamma$.
 \end{lemma}

As observed in \cite{toledo}, if
 the boundary of $\bar S$ has codimension at least three, we can apply
 a suitable weak Lefschetz theorem \cite[p 153]{gm} to show that there is a  smooth projective surface
 $Z\subset S$ with $\pi_1(Z)\cong \pi_1(S)$, and a smooth projective curve $C\subset S$
 with $\pi_1(C)\twoheadrightarrow \pi_1(S)$. By restricting $\Gamma\backslash
 \mathcal{U}$ to these varieties, and using standard facts about the
 structure of the boundary \cite[remark 2]{toledo}, \cite[thm
 4.13]{wk}, we deduce:

 \begin{lemma}
   With the notation  as above, and suppose that none of the irreducible
   factors of $\tilde S$ are the one ball, two ball, or genus $2$
   Siegel upper half plane, then $L\rtimes \Gamma\in
   \mathcal{P}$. Furthermore, there exists a fibered group $L\rtimes \Gamma_g\in
   \mathcal{P}$ where $\Gamma_g$ acts on $L$ via a surjective
   homomorphism $\Gamma_g\twoheadrightarrow\Gamma$. (For the last
   statement, it is only necessary to exclude the one ball.)
 \end{lemma}

To justify the last assertion of  proposition \ref{prop:conv}, we use Shimura curves associated
to quaternion algebras \cite{shimura}.
Let $D$ be an indefinite quaternion division algebra over $\Q$. Concretely, $D$
is a $\Q$ algebra with generators $i,j,k$ and relations $i^2=a$, $j^2=-b$
and $k= ij=-ji$ for rational numbers $a,b>0$. Either $D$ splits which
means that $D= M_2(\Q)$ or $D$ is a division algebra. We want the
latter to hold, and for this it is sufficient to assume that   the projective conic
$ax^2-by^2+z^2=0$ has no rational points. 
After extending scalars, we have an isomorphism of algebras $\psi:D\otimes
\R\cong M_2(\R)$. We have an involution on
$D$ given by  conjugation $\overline{x+y i
  +zj+wk}= x -yi -zj - wk$. Let $G = \{\alpha\in D\mid \alpha\bar
\alpha=1\}$ be viewed as an algebraic group over $\Q$. Under $\psi$ we
have $G(\R)\cong Sl_2(\R)$.  Fix a maximal order
$O \subset D$. For $\tau$ in the upper half plane $\mathbb{H}$,
$A_\tau=\C^2/\psi(O)\begin{pmatrix}\tau\\ 1\end{pmatrix}$ is an
abelian variety. For general $\tau$, its special Mumford-Tate group
is precisely $G$. If $\Gamma_1=G(\Q)\cap O$, Shimura proves that the corresponding Shimura
variety $S=\mathbb{H}_1/\Gamma_1$ is compact. (It is the moduli space of 
abelian surfaces $A_\tau$ having multiplication by $O$.)
Therefore

\begin{lemma}
   If $\Gamma\subset \Gamma_1$ is a torsion free subgroup of finite index,
$O\rtimes \Gamma\in \mathcal{P}$.

\end{lemma}

\subsection{Kodaira surfaces}
In the previous examples, we considered only the ``top''  representations
$\tau$ and $\sigma$ because  $\tau^H, \sigma^H$ would give essentially
nothing new.  By contrast, let us a  consider Kodaira
surface. This is  a smooth projective surface   admitting an everywhere
smooth map to a curve $f:X\to C$, such that the fibres are connected
with nonconstant moduli.  Let $g$ and
$q$ denote the genus of $C$ and the fibres respectively.
The fundamental group is a nontrivial  extension
\begin{equation}
  \label{eq:KodairaPi1}
1\to \Gamma_q\to \pi_1(X)\to \Gamma_g\to 1  
\end{equation}
We have an associated homomorphism $\rho:\Gamma_g\to
Mod_q\subset O(\Gamma_q)$
into the genus $q$ mapping class group. We note that this is the sole
invariant, in the sense that the extension \eqref{eq:KodairaPi1} is uniquely
determined by $\rho$. To see this, observe that by a theorem of Eilenberg and
Maclane \cite{em}, the possible extensions with outer action $\rho$ are parameterized by
$H^2(\Gamma_g, Z(\Gamma_q))$, but the centre $Z(\Gamma_q)$ is trivial
\cite[chap 1]{fm}.   
Our main theorems give  strong restrictions on the possible values of  $\im\tau^{H}\circ
\rho$. In particular, an arbitrary $\rho$ will not come from a
Kodaira surface (or any other projective manifold).  Our interest now is in seeing how big these can be.
By assumption the map $C\to M_q$ to the
moduli space of curves is nonconstant. Therefore by Torelli's theorem, it follows that the induced map
$C\to A_q$ is also nonconstant. This forces the variation of Hodge
structure $R^1f_*\Z$ to be nontrivial. In fact, it must have infinite
monodromy. If the monodromy were finite, then we could assume, after a
finite base change, that it is trivial implying that $R^1f_*\Z$ is a
trivial VHS  by \cite[thm 7.24]{schmid}.  We can apply theorem \ref{thm:main1} to strengthen the conclusion.
Therefore we have proved that:

\begin{lemma}
  For any Kodaira surface, the representation  $\tau\circ\rho:\Gamma_g\to Sp_{2q}(\Q)$ has
infinite image. In particular, $\im \rho$ is infinite. Furthermore, the identity component of
the Zariski closure of $\im \tau\circ\rho$ is a nontrivial semisimple group of symplectic Hermitian type.
\end{lemma}

In order to say more, we need a further assumption.
 Let us say that a Kodaira  surface  is {\em generic} if the image of
$\rho$ has finite index in the mapping class
group.   Although the original examples constructed by Kodaira
\cite{kodaira} are not generic, it is easy to see that
generic Kodaira surfaces exist. Let $M_{q}[n]$ be the fine moduli space of
genus $q$ curve with level $n\ge 3$ structure, and let $M_{q}[n]^*$
denote the Satake compactification. Note that the boundary
$M_{q}[n]^*-M_{q}[n]$ has codimension at least $2$ in $M_{q}[n]^*$
provided that $q>2$.
Therefore a curve $C\subset 
M_{q}[n]^*$  given as an intersection of general ample divisors would lie entirely in $M_{q}[n]$. Let $X\to C$ be the pull back
of the universal family. Then the map on
fundamental groups $\pi_1(C)\to \pi_1(M_{q}[n])$ would be surjective
by  weak Lefschetz \cite[p 153]{gm}. But
 $\pi_1(M_{q}[n])$  is a finite index subgroup of
the mapping class group.  From corollary \ref{cor:mt}, we
see that:

\begin{lemma}
Let  $X\to C$ be  generic with $q\ge 3$ and suppose that $H\subset \Gamma_q$ is a finite index normal
subgroup such that $\Gamma_q/H$ is generated by fewer than $q$ elements.
Then the image
  $\im \tau^{H}\circ \rho$  contains a Zariski dense subgroup of $\prod_\chi \G_{H,\chi}^1$.
\end{lemma}

\end{document}